\documentclass{amsart}
\usepackage{amssymb,amstext,amsmath,amscd,amsthm,amsfonts,enumerate,graphicx,latexsym,stmaryrd,multicol}
\usepackage[usenames]{color}
\usepackage[all]{xy}

\newtheorem{thm}{Theorem}[section]
\newtheorem{lem}[thm]{Lemma}
\newtheorem{prop}[thm]{Proposition}
\newtheorem{cor}[thm]{Corollary}
\theoremstyle{definition}
\newtheorem{dfn}[thm]{Definition}
\newtheorem{ques}[thm]{Question}
\newtheorem{rem}[thm]{Remark}
\newtheorem{conv}[thm]{Convention}

\newtheorem{ex}[thm]{Example}
\theoremstyle{remark}

\newtheorem*{ac}{Acknowlegments}

\numberwithin{equation}{thm}
\def\add{\operatorname{add}}
\def\ann{\operatorname{ann}}
\def\ae{\mathsf{e}^c}
\def\at{\mathsf{t}^c}
\def\C{\mathcal{C}}
\def\cm{\operatorname{CM}}
\def\cos{\mho}
\def\depth{\operatorname{depth}}
\def\E{\mathrm{E}}
\def\e{\operatorname{\mathbb{E}}}
\def\ee{\mathsf{e}}
\def\eee{\operatorname{\mathbf{E}}}

\def\Ext{\operatorname{Ext}}
\def\G{\mathcal{G}}
\def\height{\operatorname{ht}}
\def\H{\mathrm{H}}
\def\Hom{\operatorname{Hom}}
\def\id{\mathrm{id}}
\def\lend{\operatorname{\underline{End}}}
\def\lhom{\operatorname{\underline{Hom}}}

\def\m{\mathfrak{m}}
\def\Max{\operatorname{Max}}
\def\mod{\operatorname{mod}}
\def\nf{\operatorname{NF}}
\def\ng{\operatorname{NonGor}}
\def\p{\mathfrak{p}}

\def\radius{\operatorname{radius}}
\def\rank{\operatorname{rank}}
\def\rhom{\operatorname{\mathbf{R}Hom}}
\def\sing{\operatorname{Sing}}
\def\size{\operatorname{size}}
\def\spec{\operatorname{Spec}}
\def\speco{\operatorname{Spec_0}}
\def\supp{\operatorname{Supp}}
\def\syz{\mathrm{\Omega}}
\def\t{\operatorname{\mathbb{T}}} 
\def \L {\mathbf{L}} 
\def\Tor{\operatorname{Tor}}
\def\tr{\operatorname{Tr}}
\def\trace{\operatorname{tr}}
\def\ttt{\operatorname{\mathbf{T}}}
\def\V{\operatorname{V}}
\def\X{\mathcal{X}}
\def\Y{\mathcal{Y}}
 
\begin{document}
\title{Comparisons between annihilators of Tor and Ext}
\author{Souvik Dey}
\address{Department of Mathematics, University of Kansas, Lawrence, KS 66045-7523, USA}
\email{souvik@ku.edu}
\author{Ryo Takahashi}
\address{Graduate School of Mathematics, Nagoya University, Furocho, Chikusaku, Nagoya 464-8602, Japan}
\email{takahashi@math.nagoya-u.ac.jp}
\urladdr{https://www.math.nagoya-u.ac.jp/~takahashi/}
\thanks{Takahashi was partly supported by JSPS Grant-in-Aid for Scientific Research 19K03443}
\dedicatory{Dedicated to Professor Nguyen Tu Cuong on the occasion of his seventieth birthday}
\begin{abstract}
In this paper, we compare annihilators of Tor and Ext modules of finitely generated modules over a commutative noetherian ring. 
For local Cohen--Macaulay rings, one of our results refines a theorem of Dao and Takahashi.  
\end{abstract}    
\maketitle
\section{Introduction} 

Let $R$ be a $d$-dimensional Cohen--Macaulay local ring.
Let $0\le c\le d$ be an integer.
In this paper, we are mainly interested in the two ideals
$$
\at_n(R)=\bigcap\ann_R\Tor_i^R(M,N),\qquad
\ae_n(R)=\bigcap\ann_R\Ext_R^i(M,N)
$$
of $R$, where $n \ge 0$ is an integer and the intersections are taken over the integers $i>n$, and the maximal Cohen--Macaulay modules $M,N$ that are locally free in codimension less than $c$.
Namely, motivated by \cite[Theorem 1.1]{dim} which relates $\ee_0^d(R)$ and $\mathsf{t}_0^d(R)$ to the singular locus of $R$ and the dimension of the subcategory of maximal Cohen--Macaulay $R$-modules which are locally free on the punctured spectrum of $R$, we investigate when $\ae_n(R)$ and $\at_n(R)$ coincide (at least up to radical) in general.
Our first main result in this direction is the theorem below, which is included in Theorems \ref{3.12}(3), \ref{3}(2) and Corollaries \ref{13}, \ref{12}(2). 

\begin{thm}\label{14}
The following assertions hold true.
\begin{enumerate}[\rm(1)]
\item
If $R$ is either artinian or Gorenstein, then $\at_n(R)=\ae_n(R)$.
\item
If $R$ is either locally Gorenstein in codimension less than $c$ and admits a canonical module or a complete equicharacteristic local ring with perfect residue field, then $\sqrt{\at_n(R)}=\sqrt{\ae_n(R)}$.  
\end{enumerate}
\end{thm}

Denote by $\mod R$ the category of finitely generated $R$-modules, and by $\cm^c(R)$ the full subcategory of $\mod R$ consisting of maximal Cohen--Macaulay modules that are locally free in codimension less than $c$. 
As an application of Theorem \ref{14} we obtain the following result, which is the same as Corollary \ref{8}.

\begin{thm}\label{15} 
Consider the following conditions.
\begin{quote}
{\rm(i)}\ $\cm^c(R)$ has finite dimension.\qquad
{\rm(ii)}\ $\height\ae_n(R)\ge c$.\qquad
{\rm(iii)}\ $\height\at_n(R)\ge c$.\\
{\rm(iv)}\ $R$ is locally regular in codimension less than $c$.
\end{quote}
\begin{enumerate}[\rm(1)]
\item
The implications ${\rm(i)}\Rightarrow{\rm(ii)}\Leftrightarrow{\rm(iii)}\Rightarrow{\rm(iv)}$ hold when $R$ admits a canonical module.
\item
The implications ${\rm(i)}\Leftrightarrow{\rm(ii)}\Leftrightarrow{\rm(iii)}\Rightarrow{\rm(iv)}$ hold when $c=d$.    
\item
The implications ${\rm(i)}\Leftrightarrow{\rm(ii)}\Leftrightarrow{\rm(iii)}\Leftrightarrow{\rm(iv)}$ hold when $R$ is excellent, equicharacteristic and admits a canonical module. 
\end{enumerate}
\end{thm}

When $n=0$, the main result of \cite{dim} (i.e. \cite[Theorem 1.1]{dim}) asserts Theorem \ref{15}(2) minus the implication (iii) $\Rightarrow$ (ii), and Theorem \ref{15}(3) for $c=d$ under the assumption that $R$ is complete and has perfect coefficient field.
Thus, Theorem \ref{15} highly refines the main result of \cite{dim}.  

The organization of this paper is as follows.
In Section 2, we state definitions and properties of fundamental notions used in this paper.
In Section 3, we explore annihilators of Tor and Ext modules over a general commutative noetherian ring. In Section 4, we focus on annihilators of Tor and Ext modules over a Cohen--Macaulay local ring.

\section{Preliminaries}

What we state in this section is used in the next sections.
We begin with our convention.

\begin{conv}
Throughout the present paper, let $R$ be a commutative noetherian ring.
We assume that all modules are finitely generated and all subcategories are strictly full.
Denote by $\mod R$ the category of (finitely generated) $R$-modules, and by $\cm(R)$ the subcategory of $\mod R$ consisting of maximal Cohen--Macaulay modules (recall that an $R$-module $M$ is called {\em maximal Cohen--Macaulay} if $\depth_{R_\p}M_\p=\dim R_\p$ for all $\p\in\supp_RM$).
We identify each $R$-module $M$ with the subcategory $\{M\}$ of $\mod R$ consisting of $M$ (since we assume all subcategories are strictly full, $\{M\}$ denotes the full subcategory of $\mod R$ consisting of those modules which are isomorphic to $M$).
We denote by $(-)^\ast$ the algebraic dual $\Hom_R(-,R)$. 
Whenever $R$ is a local ring, $(-)^\vee$ stands for the Matlis dual.
Whenever $R$ is a local Cohen--Macaulay ring with a canonical module $\omega$, we denote by $(-)^\dag$ the canonical dual $\Hom_R(-,\omega)$. 
\end{conv}

From now on, we state the definitions of notions used in the next sections together with a couple of their basic properties.

\begin{dfn}
Let $\X$ be a subcategory of $\mod R$.
\begin{enumerate}[(1)]
\item
We denote by $\add\X$ the subcategory of $\mod R$ consisting of direct summands of finite direct sums of modules in $\X$.
Note that $\add R$ equals the subcategory of $\mod R$ consisting of projective modules.
\item
We denote by $\underline\X$ the subcategory of $\mod R$ consisting of modules $M$ such that there is an isomorphism $M\oplus P\cong X\oplus Q$ with $P,Q\in\add R$ and $X\in\X$.
We say that $\X$ is {\em stable} if $\underline\X=\X$.
Note that $\underline {\underline \X}=\underline \X$, so that $\underline \X$ is stable.
Also, $\add R$ and $\mod R$ are stable.
\item
Suppose that $R$ is a local Cohen--Macaulay ring with a canonical module $\omega$ and that $\X$ is contained in $\cm(R)$.
We denote by $\overline\X$ the subcategory of $\cm(R)$ consisting of maximal Cohen--Macaulay modules $M$ such that there is an isomorphism $M\oplus I\cong X\oplus J$ with $I,J\in\add\omega$ and $X\in\X$.
We say that $\X$ is {\em costable} if $\overline\X=\X$.
Note that $\add\omega$ and $\cm(R)$ are costable. 
\end{enumerate}
\end{dfn} 

\begin{dfn}
For an $R$-module $M$, we denote by $\syz M$ the {\em (first) syzygy} of $M$, that is, the kernel of an epimorphism from a projective $R$-module to $M$.
For $n\ge1$ we inductively define the {\em $n$th syzygy} of $M$ by $\syz^nM=\syz(\syz^{n-1}M)$, and set $\syz^0M=M$.
The $n$th syzygy of $M$ is uniquely determined up to projective summands.
For a subcategory $\X$ of $\mod R$ and an integer $n\ge0$, we set $\syz^n\X=\underline{\{\syz^nX\mid X\in\X\}}$.
We say that a subcategory $\X$ of $\mod R$ is {\em closed under syzygies} if $\syz\X\subseteq\X$.
Note that $\syz\X=\syz\underline \X=\underline{\syz \X}$.
\end{dfn}

\begin{dfn}
For an $R$-module $M$ we denote by $\tr M$ the {\em (Auslander) transpose} of $M$.
This is defined as follows.
Take a projective presentation $P_1\xrightarrow{f}P_0\to M\to 0$.
Dualizing this by $R$, we get an exact sequence $0\to M^\ast\to P_0^\ast\xrightarrow{f^\ast}P_1\to\tr M\to0$, that is, $\tr M$ is the cokernel of the map $f^\ast$.
The transpose of $M$ is uniquely determined up to projective summands; see \cite{AB} for basic properties.
For a subcategory $\X$ of $\mod R$, we set $\tr\X=\underline{\{\tr X\mid X\in\X\}}$.
We say that $\X$ is {\em closed under transposes} if $\tr\X\subseteq\X$.
Note that there are equalities $\tr\X=\tr\underline \X=\underline{\tr\X}$.
\end{dfn}

\begin{dfn}
Let $\Phi$ be a subset of $\spec R$.
We define the {\em (Krull) dimension} of $\Phi$ by $\dim\Phi=\sup\{\dim R/\p\mid\p\in\Phi\}$.
\end{dfn}

\begin{dfn}
For an $R$-module $M$ we denote by $\nf(M)$ the {\em nonfree locus} of $M$, that is, the set of prime ideals $\p$ of $R$ such that the $R_\p$-module $M_\p$ is nonfree.
It is well-known and easy to see that $\nf(M)$ is a closed subset of $\spec R$ in the Zariski topology.
We set $\speco(R)=\spec R\setminus\Max R$ and call it the {\em punctured spectrum} of $R$.
Note that an $R$-module $M$ is locally free on $\speco(R)$ if and only if $\dim\nf(M)\le0$.
For each $n\ge0$, we denote by $\mod_{n}(R)$ the subcategory of $\mod R$ consisting of modules $M$ such that $\dim\nf(M)\le n$.
It is easy to see that $\mod_n(R)$ is stable and closed under syzygies and transposes. 
For an integer $n\ge0$ we set $\cm_{n}(R)=\cm(R)\cap\mod_{n}(R)$.
Note that $\cm_n(R)$ is stable and closed under syzygies if $R$ is Cohen--Macaulay.
\end{dfn}

\begin{dfn}
An $R$-module $M$ is said to be {\em totally reflexive} if $\Ext_R^i(M,R)=\Ext_R^i(\tr M,R)=0$ for all $i>0$.
This is equivalent to saying that the canonical map $M\to M^{\ast\ast}$ is an isomorphism (i.e., $M$ is reflexive) and $\Ext_R^i(M,R)=\Ext_R^i(M^\ast,R)=0$ for all $i>0$.
Every totally reflexive module is the syzygy of some totally reflexive module.
If $R$ is Cohen--Macaulay, then every totally reflexive $R$-module is maximal Cohen--Macaulay.
Also, $R$ is Gorenstein if and only if every maximal Cohen--Macaulay $R$-module is totally reflexive.
For more details of totally reflexive modules, we refer the reader to \cite{AB,C}.
We denote by $\G(R)$ the subcategory of $\mod R$ consisting of totally reflexive modules.
Note that $\G(R)$ is stable and closed under syzygies and transposes.
For an integer $n\ge0$ we set $\G_n(R)=\G(R)\cap\mod_n(R)$, which is stable and closed under syzygies and tranposes as well.
If $R$ is Gorenstein, then $\G_n(R)=\cm_n(R)$.
\end{dfn} 

\begin{dfn}
Let $R$ be a local Cohen--Macaulay ring with a canonical module $\omega$. 
\begin{enumerate}[(1)]
\item
For a maximal Cohen--Macaulay $R$-module $M$, we denote by $\cos M$ the {\em (first) cosyzygy} of $M$, that is, the maximal Cohen--Macaulay cokernel of a monomorphism to a module in $\add\omega$.
Cosyzygies always exist: taking an exact sequence $0\to\syz(M^\dag)\to R^{\oplus n}\to M^\dag\to0$ and dualizing it by $\omega$, one obtains an exact sequence $0\to M\to \omega^{\oplus n}\to(\syz(M^\dag))^\dag\to0$ of maximal Cohen--Macaulay modules, and then $(\syz(M^\dag))^\dag$ is the cosyzygy of $M$.
For $n\ge1$ we inductively define the {\em $n$th cosyzygy} of $M$ by $\cos^nM=\cos(\cos^{n-1}M)$, and set $\cos^0M=M$.
Note that the $n$th cosyzygy of $M$ is uniquely determined up to direct summands that belong to $\add\omega$. 
For a subcategory $\X$ of $\cm(R)$ and an integer $n\ge0$, we set $\cos^n\X=\overline{\{\cos^nX\mid X\in\X\}}$.
We say that $\X$ is {\em closed under cosyzygies} if $\cos\X\subseteq\X$.
\item
For a subcategory $\X$ of $\cm(R)$ we set $\X^\dag=\{X^\dag\mid X\in\X\}$.
Then
$$
\X^\dag\subseteq\X
\iff\X\subseteq\X^\dag
\iff\X=\X^\dag.
$$
We say that $\X$ is {\em closed under canonical duals} if one of these equivalent conditions holds.
\end{enumerate}
\end{dfn}

\begin{rem}
\begin{enumerate}[(1)]
\item
Suppose that a subcategory $\X$ of $\mod R$ is closed under syzygies.
Then for an exact sequence $0\to M\to P\to X\to0$ of $R$-modules with $X\in\X$ and $P\in\add R$, one has $M\in\X$.
The converse holds if $\X$ is stable.
\item
Suppose that a subcategory $\X$ of $\mod R$ is closed under transposes.
Then for an exact sequence $P_1\xrightarrow{f}P_0\to X\to0$ of $R$-modules with $X\in\X$ and $P_0,P_1\in\add R$, the cokernel of $f^\ast$ belongs to $\X$.
The converse holds if $\X$ is stable.
\item
Let $R$ be a local Cohen--Macaulay ring with a canonical module $\omega$.
Let $\X$ be a subcategory of $\cm(R)$.
Suppose that $\X$ is closed under cosyzygies.
Then for an exact sequence $0\to X\to I\to M\to0$ of maximal Cohen--Macaulay $R$-modules with $X\in\X$ and $I\in\add\omega$, one has $M\in\X$.
The converse holds if $\X$ is costable.
\end{enumerate} 
\end{rem} 

Now we recall the following notions from \cite{radius}.

\begin{dfn}
Let $R$ be a local ring
\begin{enumerate}[\rm(1)]
\item
For a subcategory $\X$ of $\mod R$, we put
$$
|\X|=\add \X,\qquad
[\X]=\add(\{R\}\cup\{\syz^i X\mid i\ge 0,X\in\X\}).
$$
\item
For subcategories $\X, \Y$ of $\mod R$, we denote by $\X \circ \Y$ the subcategory of $\mod R$ consisting of the $R$-modules $M$ which fits into an exact sequence $0\to X \to M \to Y \to 0$ with $X\in \X$ and $Y\in \Y$.
We set $\X * \Y=||\X| \circ |\Y||$ and $\X \bullet \Y=[[\X] \circ [\Y]]$.
\item
For a subcategory $\X$ of $\mod R$ and integer $r\ge 1$, inductively define $|\X|_r$ and $[\X]_r$ respectively, as follows.
$$
|\X|_r=\begin{cases}0&\text{ if }r=0,\\|\X| \quad  &\text{ if } r=1, \\ |\X|_{r-1}* \X \quad &\text { if } r\ge 2,
\end{cases}
\qquad
[\X]_r=\begin{cases}0&\text{ if }r=0,\\ [\X] \quad  &\text{ if } r=1, \\ [\X]_{r-1}\bullet \X \quad &\text { if } r\ge 2.
\end{cases}
$$
\item
Let $\X$ be a subcategory of $\mod R$.
The {\em radius} (resp. {\em dimension}) of $\X$, which is denoted by $\radius\X$ (resp. $\dim\X$), is by definition the infimum of the integers $n\ge 0$ such that $\X \subseteq [G]_{n+1}$ (resp. $\X = [G]_{n+1}$) for some $G\in \mod R$.
Similarly, the {\em size} (resp. {\em rank}) of $\X$, which is denoted by $\size\X$ (resp. $\rank\X$), is defined to be the infimum of the integers $n\ge 0$ such that $\X \subseteq |G|_{n+1}$ (resp. $\X =|G|_{n+1}$) for some $G\in \mod R$.
By definition, the radius, dimension, size and rank take values in $\mathbb N \cup \{\infty\}$.
Moreover, since $|\X|_n \subseteq [\X]_n$ for all $n\ge 1$, it holds that $\dim \X \ge \radius\X \le\size\X \le\rank\X$.
If $\X$ is moreover resolving, then $\dim \X \le \rank\X$. 
\end{enumerate}
\end{dfn}

\section{Annihilators over a commutative noetherian ring}

In this section we investigate annihilators of Tor and Ext over an arbitrary commutative noetherian ring $R$.
First of all, we give their definitions.

\begin{dfn}
Let $\X,\Y$ be subcategories of $\mod R$, and let $n\ge0$ be an integer.
We set

$$
\ttt_n(\X,\Y)=\bigoplus_{i>n}\bigoplus_{X\in\X}\bigoplus_{Y\in\Y}\Tor_i^R(X,Y),\qquad
\eee^n(\X,\Y)=\bigoplus_{i>n}\bigoplus_{X\in\X}\bigoplus_{Y\in\Y}\Ext^i_R(X,Y).
$$
We define the ideals $\t_n(\X,\Y)$ and $\e^n(\X,\Y)$ of $R$ by
\begin{align*}
\t_n(\X,\Y)&=\ann_R\ttt_n(\X,\Y)=\bigcap_{i>n}\bigcap_{X\in\X}\bigcap_{Y\in\Y}\ann_R\Tor_i^R(X,Y),\\
\e^n(\X,\Y)&=\ann_R\eee^n(\X,\Y)=\bigcap_{i>n}\bigcap_{X\in\X}\bigcap_{Y\in\Y}\ann_R\Ext_R^i(X,Y).
\end{align*}
Note that if $\X$ is closed under syzygies, $\t_n(\X,\Y)=\bigcap_{X\in\X}\bigcap_{Y\in\Y}\ann_R\Tor_{n+1}^R(X,Y)$ and $\e^n(\X,\Y)=\bigcap_{X\in\X}\bigcap_{Y\in\Y}\ann_R\Ext_R^{n+1}(X,Y)$.
We put $\t_n(\X)=\t_n(\X,\X)$ and $\e^n(\X)=\e^n(\X,\X)$.
\end{dfn}

\begin{rem}
Let $\X,\Y$ be subcategories of $\mod R$ and $n\ge0$ an integer.
Then it is easy to observe that
$$
\supp_R(\ttt_n(\X,\Y))\subseteq\V(\t_n(\X,\Y)),\qquad
\supp_R(\eee^n(\X,\Y))\subseteq\V(\e^n(\X,\Y)).
$$
As we will see in Remark \ref{10}(3), these inclusions are not necessarily equalities.
Thus, it is not sufficient to investigate the supports of the modules $\ttt_n(\X,\Y)$ and $\eee^n(\X,\Y)$ to get the structure of (the radicals of) the ideals $\t_n(\X,\Y)$ and $\e^n(\X,\Y)$.
\end{rem}

Let $M,N$ be $R$-modules.
We denote by $\lhom_R(M,N)$ the quotient of $\Hom_R(M,N)$ by homomorphisms $M\to N$ factoring through projective $R$-modules.
We set $\lend_R(M)=\lhom_R(M,M)$.
The following lemma yields an isomorphism between Tor and Ext modules.

\begin{lem}\label{16}
For $R$-modules $M,N$ one has an isomorphism $\Ext_R^1(\tr\syz\tr\syz M,N)\cong\Tor_1^R(\tr\syz M,N)$.
\end{lem}

\begin{proof}
An exact sequence $0\to\Ext_R^1(\tr\syz\tr\syz M,N)\to\Tor_1^R(\tr\syz M,N)\to\Hom_R(\Ext_R^1(\tr\syz M,R),N)$ exists by \cite[Theorem (2.8)]{AB}.
As $\syz M$ is $1$-torsionfree (\cite[Definition (2.15)]{AB}), we have $\Ext_R^1(\tr\syz M,R)=0$.
Thus the isomorphism in the lemma is obtained.
\end{proof}

\begin{rem}
Here is another proof of Lemma \ref{16}, which may be easier for the reader who is familiar with Auslander's approximation theory:
There is an exact sequence $0\to\syz M\xrightarrow{f}P\to\tr\syz\tr\syz M\to0$ such that $f$ is a left $(\add R)$-approximation. 
An exact sequence $\Hom_R(P,N)\xrightarrow{f'}\Hom_R(\syz M,N)\to\Ext_R^1(\tr\syz\tr\syz M,N)\to0$ is induced.
We can verify that the image of $f'$ coincides with the set of homomorphisms $\syz M\to N$ factoring through projective modules.
We obtain an isomorphism $\Ext_R^1(\tr\syz\tr\syz M,N)\cong\lhom_R(\syz M,N)$.
Combining this with \cite[Lemma (3.9)]{Y} deduces the lemma.
\end{rem}

Using the above lemma, we obtain the following proposition on annihilators.

\begin{prop}\label{19}
Let $\X$ be a subcategory of $\mod R$.
Suppose that $\X$ is contained in $\syz(\mod R)$, and closed under syzygies and transposes.
Then for all integers $n\ge0$ and subcategories $\Y$ of $\mod R$ one has $\t_n(\X,\Y)=\e^n(\X,\Y)$.
\end{prop}

\begin{proof}
Note that $\e^n(\X,\Y)=\e^n(\underline \X,\Y)$ and $\t_n(\X,\Y)=\t_n(\underline \X,\Y)$.
Replacing $\X$ by $\underline\X$, we may assume $\X$ is stable.
Any $R$-module $M$ satisfies $\tr\tr M\cong M$ up to projective summands by \cite[Proposition (2.6)]{AB}, which implies $\X=\tr\X$.
Let $X\in\X$.
Then $X$ is a syzygy, and hence it is $1$-torsionfree by \cite[Theorem (2.17)]{AB}.
Therefore $X\cong\syz\tr\syz\tr X$ up to projective summands by \cite[Theorem (2.17)]{AB} again.
As $\X$ is closed under syzygies and transposes, we get $\tr\syz\tr X\in\X$, and $X\in\syz\X$.
Thus $\X=\syz\X$, and $\X=\tr\X=\tr\syz\X$.
There are equalities
\begin{align}
\label{17}\t_n(\X,\Y)&=\t_0(\syz^n\X,\Y)=\t_0(\X,\Y)=\textstyle\bigcap_{X\in\X,Y\in\Y}\ann_R\Tor_1^R(X,Y),\\
\label{18}\e^n(\X,\Y)&=\e^0(\syz^n\X,\Y)=\e^0(\X,\Y)=\textstyle\bigcap_{X\in\X,Y\in\Y}\ann_R\Ext_R^1(X,Y).
\end{align}
Using the equalities $\X=\tr\syz\X=\tr\syz\tr\syz\X$ and Lemma \ref{16}, we observe that the last terms in \eqref{17} and \eqref{18} coincide.
\end{proof}

The following corollary is a direct consequence of Proposition \ref{19}, which gives part of Theorem \ref{14}(1).

\begin{cor}\label{13}
Let $n,t\ge0$ be any integers and let $\C$ be any subcategory of $\mod R$.
Then the equality $\t_n(\G_t(R),\C)=\e^n(\G_t(R),\C)$ holds.
In particular, $\t_n(\cm_t(R))=\e^n(\cm_t(R))$ if $R$ is Gorenstein.
\end{cor} 

The following example illustrates the usefulness of Corollary \ref{13}, where it is easy to compute the annihilator of $\Tor$, and hence we can conclude what is the annihilator of $\Ext$.

\begin{ex}
Let $k$ be a field and consider the hypersurface $R=k[\![x,y]\!]/(x^2)$.
By virtue of \cite[Proposition 4.1]{BGS}, we have that
$$
R, (x), (x,y), (x,y^2), (x,y^3),\dots
$$
give a complete list of isomorphism classes of indecomposable maximal Cohen--Macaulay $R$-modules.
Among these, all the modules except $(x)$ are locally free on punctured spectrum of $R$.
We get
\begin{align*}
\t_0(\cm_0(R))
&\textstyle=\bigcap_{h>0}\bigcap_{M,N\in\cm_0(R)}\ann_R\Tor_h^R(M,N)
=\bigcap_{h,i,j>0}\ann_R\Tor_h^R(I_i,I_j)\\
&\textstyle=\bigcap_{h>0,i\ge j>0}\ann_R\Tor_h^R(I_i,I_j)
=\bigcap_{h>0,i\ge j>0}\ann_R\Tor_{h+1}^R(I_i,R/I_j),
\end{align*}
where we set $I_i=(x,y^i)$ for each $i>0$.
The minimal free resolution of the $R$-module $I_i$ is
$$
F(i)=(\cdots\xrightarrow{\left(\begin{smallmatrix}x&y^i\\0&-x\end{smallmatrix}\right)}R^{\oplus2}\xrightarrow{\left(\begin{smallmatrix}x&y^i\\0&-x\end{smallmatrix}\right)}R^{\oplus2}\xrightarrow{\left(\begin{smallmatrix}x&y^i\\0&-x\end{smallmatrix}\right)}R^{\oplus2}\to0).
$$
For any integers $i\ge j>0$, a complex
$$
F(i)\otimes_RR/I_j=(\cdots\xrightarrow{0}(R/I_j)^{\oplus2}\xrightarrow{0}(R/I_j)^{\oplus2}\xrightarrow{0}(R/I_j)^{\oplus2}\to0)
$$
is induced, which shows that $\Tor_p^R(I_i,R/I_j)\cong\H_p(F(i)\otimes_RR/I_j)\cong(R/I_j)^{\oplus 2}$ for all integers $p\ge0$ and $i\ge j>0$.
Thus we obtain
$$
\textstyle\t_0(\cm_0(R))=\bigcap_{h>0,i\ge j>0}\ann_R(R/I_j)^{\oplus2}=\bigcap_{j>0}I_j=(x),
$$
where the last equality is a consequence of Krull's intersection theorem.
Hence, by Corollary \ref{13}, we obtain $\e^0(\cm_0(R))=(x)$.
Note that indeed this can be obtained by a direct computation for Ext without appealing to our Corollary \ref{13}, but the computation is more complicated; in the above computation we use the general isomorphism $\Tor_h^R(I_i,I_j)\cong\Tor_{h+1}^R(I_i,R/I_j)$, while in general $\Ext^h_R(I_i,I_j)\ncong\Ext^{h+1}_R(I_i,R/I_j)$.
\end{ex}

To prove our next proposition, we establish a lemma.

\begin{lem}\label{1}
Let $M$ be an $R$-module.
Then
$$\textstyle
\t_0(M,\mod R)=\ann_R\lend_R(M)=\ann_R\Ext_R^1(M,\syz M)
=\e^0(M,\mod R).
$$ 
\end{lem}

\begin{proof}
We call the four ideals (1), (2), (3) and (4) in order.
Clearly, (3) contains (4).
Let $a$ be an element of (3).
Then the multiplication map $M\xrightarrow{a}M$ factors through a projective module by the proof of \cite[Lemma 2.14]{ua}.
Hence $M\xrightarrow{a}M$ is zero in $\lend_R(M)$, and so is the composition of any endomorphism of $M$ with $M\xrightarrow{a}M$.
Thus (2) contains (3).
Let $b$ be an element of (2).
Then $b\cdot\id_M$ is zero in $\lend_R(M)$, which means that the multiplication map $M\xrightarrow{b}M$ factors through a projective module.
There is a diagram $M\xrightarrow{f}P\xrightarrow{g}M$ of homomorphisms of $R$-modules with $P$ projective such that $gf=(M\xrightarrow{b}M)$.
Applying $\Tor^R_i(-,N)$ and $\Ext_R^i(-,N)$ with $i>0$ and $N\in\mod R$, we see that the multiplication maps $\Tor^R_i(M,N)\xrightarrow{b}\Tor^R_i(M,N)$ and $\Ext_R^i(M,N)\xrightarrow{b}\Ext_R^i(M,N)$ are zero as $\Tor^R_i(P,N)=\Ext_R^i(P,N)=0$.
Thus (1) and (4) contain (2).
There is an isomorphism $\lend_R(M)\cong\Tor_1^R(M,\tr M)$ by \cite[Lemma (3.9)]{Y}, from which we see that (2) contains (1).
\end{proof} 

Now we obtain inclusions among annhilators of Tor and Ext.

\begin{prop}\label{7}
Let $\X,\Y$ be subcategories of $\mod R$.
Assume that $\X$ is closed under syzygies and that $\Y$ contains $\X$.
Then for any integer $n\ge0$ one has
\begin{align*}
\e^n(\X)
&\textstyle=\bigcap_{X\in\X}\ann_R\Ext_R^{n+1}(X,\syz^{n+1}X)\\
&\textstyle=\bigcap_{X\in\X}\ann_R\lend_R(\syz^nX)
=\e^n(\X,\Y)\subseteq\t_n(\X,\Y)\subseteq\t_n(\X).
\end{align*}
If moreover $\Y$ contains $\tr\X$, then one has the equality $\t_n(\X,\Y)=\e^n(\X,\Y)$.
So in particular, if $\X$ is closed under syzygies and transposes, then $\e^n(\X)=\t_n(\X)$ holds.
\end{prop}

\begin{proof}
As $\Y$ contains $\X$, the ideals $\e^n(\X)$ and $\t_n(\X)$ contain $\e^n(\X,\Y)$ and $\t_n(\X,\Y)$, respectively.
Since $\X$ is closed under syzygies, we have an equality $\e^n(\X)=\bigcap_{X,X'\in\X}\ann_R\Ext_R^{n+1}(X,X')$, the right-hand side of which is contained in $I:=\bigcap_{X\in\X}\ann_R\Ext_R^{n+1}(X,\syz^{n+1}X)$ as $\syz^{n+1}X\in\X$.
Lemma \ref{1} implies
\begin{align*}
I&\textstyle=\bigcap_{X\in\X}\ann_R\lend_R(\syz^nX)
=\bigcap_{X\in\X}\e^0(\syz^nX,\mod R)
=\e^n(\X,\mod R)\subseteq\e^n(\X,\Y),\\
I&\textstyle=\bigcap_{X\in\X}\ann_R\lend_R(\syz^nX)\subseteq\bigcap_{X\in\X}\t_0(\syz^nX,\mod R)=\t_n(\X,\mod R)\subseteq\t_n(\X,\Y).
\end{align*}
Now the proof of the first assertion of the proposition is completed.

Next we show the last assertion of the proposition.
We already know that $\t_n(\X,\Y)$ contains $\e^n(\X,\Y)$.
Let $a\in\t_n(\X,\Y)$ and $X\in\X$.
The assumption implies $\syz^nX\in\X$ and $\tr\syz^nX\in\Y$.
The isomorphisms
$$
\lend_R(\syz^nX)
\cong\Tor_1^R(\tr\syz^nX,\syz^nX)
\cong\Tor_1^R(\syz^nX,\tr\syz^nX)
\cong\Tor_{n+1}^R(X,\tr\syz^nX)
$$
hold, where the first isomorphism follows from \cite[Lemma (3.9)]{Y}.
The last term $\Tor_{n+1}^R(X,\tr\syz^nX)$ is annihilated by the element $a$, and so is the first term $\lend_R(\syz^nX)$.
It follows that $\t_n(\X,\Y)$ is contained in $\bigcap_{X\in\X}\ann_R\lend_R(\syz^nX)=\e^n(\X,\Y)$.
\end{proof} 

\begin{rem}
Proposition \ref{7} also deduces Corollary \ref{13} for those $\C$ which contain $\G_t(R)$.
\end{rem}

To derive our next annihilator relations, we need a lemma, generalizing \cite[Propositions 4.5 and 4.6(1)]{dim}.
Assertions (1) and (2) are shown similarly as in the proof of \cite[Proposition 4.5]{dim}, while (3) is deduced analogously as in the proof of \cite[Proposition 4.6(1)]{dim} by using (1) and (2) instead of \cite[Proposition 4.5]{dim}.

\begin{lem}\label{3.11}
Let $R$ be a $d$-dimensional Cohen--Macaulay local ring.
Let $n\ge 0$ be an integer.  
\begin{enumerate}[\rm(1)]
\item
Let $a\in \t_n(\cm_0(R))$.
Then $a^{2^{2d}}\Tor_i^R(M,N)=0$ for all $i>n+4d$ and $M,N\in \mod R$.
\item
Let $a\in \e^n(\cm_0(R))$.
Then $a^{2^{2d}(d+1)}\Ext^i_R(M,N)=0$ for all $i>n+d$ and $M,N\in \mod R$.
\item
It holds that $\sing R\subseteq \V(\t_{n+4d}(\mod R))\cap \V(\e^{n+d}(\mod R))\subseteq \V(\t_n(\cm_0(R))) \cap \V(\e^n(\cm_0(R)))$.
\end{enumerate}
\end{lem}

Now we can prove the following theorem, which gives a $\Tor$ version of \cite[Proposition 4.8]{dim}, and contains part of Theorem \ref{14}(2).  

\begin{thm}\label{3.12}
Let $(R,\m,k)$ be a equicharacteristic, local Cohen--Macaulay ring of dimension $d$.
Let $n\ge 0$ be an integer.
Let $\cm_0(R)\subseteq \X \subseteq \Y$ be subcategories of $\mod R$, where $\X$ is closed under syzygies.
Then $\sing R=\V(\t_n(\X,\Y))=\V(\e^n(\X,\Y))$, if one of the following three conditions is satisfied.
\begin{align*}
\text{{\rm(1)} $R$ is excellent and $n\ge 2d$.}\quad
&\text{{\rm(2)} $R$ is complete, $k$ is perfect and $n\ge d$.}\\
&\text{{\rm(3)} $R$ is complete, $k$ is perfect and $\X \subseteq \cm(R)$.}
\end{align*}
\end{thm}

\begin{proof}
Applying Lemma \ref{3.11}(3) and Proposition \ref{7} gives rise to inclusions $\sing R \subseteq \V(\t_n(\cm_0(R)))\subseteq \V(\t_n(\X,\Y))\subseteq \V(\e^n(\X,\Y))$.
It remains to show that $\V(\e^n(\X,\Y))$ is contained in $\sing R$.

(1) $\V(\e^n(\X,\Y))$ is contained in $\V(\e^{2d}(\mod R))$, which is equal to $\sing R$ by \cite[Theorem 5.3]{ua}.

(2) $\V(\e^n(\X,\Y))$ is contained in $\V(\e^{d}(\mod R))$, which is contained in  $\sing R$ by \cite[Corollary 5.15]{W}.   

(3) $\V(\e^n(\X,\Y))$ is contained in $\V(\e^0(\X,\Y))$, which is equal to $\sing R$ by \cite[Proposition 4.8]{dim}.
\end{proof} 

\section{Annihilators over a Cohen--Macaulay local ring}

In this section, we consider annihilators of Tor and Ext modules over a Cohen--Macaulay local ring.

Let $M$ be an $R$-module.
The {\em trace ideal} of $M$, denoted $\trace M$, is defined by the image of the canonical map $\Hom_R(M,R)\otimes_RM\to R$ given by $f\otimes x\mapsto f(x)$.
To prove the proposition below, we establish a lemma.

\begin{lem}\label{2}
Let $R$ be a $d$-dimensional Cohen--Macaulay local ring with a canonical module $\omega$.
Let $M,N$ be $R$-modules.
Let $i\ge1$ be an integer. 
Suppose that $N$ is maximal Cohen--Macaulay.

\begin{enumerate}[\rm(1)]
\item
Let $t$ be an integer with $\dim\nf(M)\le t\le d$.
Then one has
\begin{align*}
\textstyle\prod_{j=0}^t\ann_R\Ext_R^{d-j}(\Tor_{i+j}^R(M,N),\omega)&\subseteq\ann_R\Ext_R^{d+i}(M,N^\dag),\\
\textstyle\prod_{j=0}^t\ann_R\Ext_R^{d-j}(\Ext_R^{d+i-j}(M,N),\omega)&\subseteq\ann_R\Tor_i^R(M,N^\dag).
\end{align*} 
If $t=0$, then $\ann_R\Tor_i^R(M,N)=\ann_R\Ext_R^{d+i}(M,N^\dag)$.
\item
For any integer $r\ge0$, one has
$$
(\trace\omega)^r\cdot\ann_R\Ext_R^{r+i}(N,\syz^rM)\subseteq \ann_R\Ext_R^i(N,M)\supseteq(\trace\omega)^r\cdot\ann_R\Ext_R^{r+i}(\cos^rN,M).
$$
\end{enumerate}
\end{lem}

\begin{proof}
(1) We have an isomorphism $\rhom_R(M\otimes_R^{\L} N, \omega)\cong \rhom_R(M,\rhom_R(N,\omega))$; see \cite[(A.4.21)]{C}.
Since $N$ is maximal Cohen--Macaulay, we have $\Ext^{>0}_R(N,\omega)=0$, and $\rhom_R(N,\omega)\cong N^\dag$.
Hence $\rhom_R(M\otimes_R^{\L} N, \omega)\cong \rhom_R(M, N^{\dag})$.
This induces a spectral sequence
$$
\E_2^{pq}=\Ext_R^p(\Tor_q^R(M,N),\omega)\implies\H^{p+q}=\Ext_R^{p+q}(M,N^\dag).
$$
Clearly, $\E_2^{pq}=0$ if $p<0$ or $q<0$.
As $\omega$ has injective dimension $d$, we have that $\E_2^{pq}=0$ if $p>d$.
The support of $\Tor_q^R(M,N)$ is contained in $\nf(M)$ if $q>0$.
Local duality (\cite[Corollary 3.5.11(a)]{BH}) shows that $\E_2^{pq}=0$ if $q>0$ and $p<d-t$.
The filtration induced from the spectral sequence is
$$
\H^{d+i}=\cdots=\H^{d+i}_{d-t}\supseteq\H^{d+i}_{d-t+1}\supseteq\cdots\supseteq\H^{d+i}_d\supseteq\H^{d+i}_{d+1}=\cdots=0
$$
with $\H^{d+i}_{d-j}/\H^{d+i}_{d-j+1}=\E_\infty^{d-j,i+j}$ for each $0\le j\le t$.
In general, $\E_r^{pq}$ is a subquotient of $\E_{r-1}^{pq}$ for all $p,q,r$, and hence $\ann_R\E_\infty^{pq}$ contains $\ann_R\E_2^{pq}$ for all $p,q$.
The filtration shows that $\prod_{j=0}^t\ann_R\E_2^{d-j,i+j}\subseteq\prod_{j=0}^t\ann_R\E_\infty^{d-j,i+j}\subseteq\ann_R\H^{d+i}$.
The first inclusion in the assertion follows from this.
The second inclusion is deduced by a dual argument.
Namely, since $\omega$ has finite injective dimension, we have an isomorphism $ \rhom_R(\rhom_R(M, N), \omega)\cong M \otimes_R^{\L}\rhom_R(N,\omega)$; see \cite[(A.4.24)]{C}.
As $N$ is maximal Cohen--Macaulay, we get $\rhom_R(N,\omega)\cong N^\dag$ and $\rhom_R(\rhom_R(M, N), \omega)\cong M \otimes_R^{\L} N^{\dag}$.
This induces a spectral sequence 
$$
\E_2^{pq}=\Ext_R^p(\Ext_R^{-q}(M,N),\omega)\implies\H^{p+q}=\Tor^R_{-p-q}(M,N^\dag),
$$
and $\E_2^{pq}=0$ if (i) $p<0$, or (ii) $q>0$, or (iii) $p>d$, or (iv) $q<0$ and $p<d-t$. 
A filtration $\H^{-i}=\cdots=\H^{-i}_{d-t}\supseteq\H^{-i}_{d-t+1}\supseteq\cdots\supseteq\H^{-i}_d\supseteq\H^{-i}_{d+1}=\cdots=0$ is induced, where $\H^{-i}_{d-j}/\H^{-i}_{d-j+1}=\E_\infty^{d-j,-d-i+j}$ for each $0\le j\le t$.
This shows $\prod_{j=0}^t\ann_R\E_2^{d-j,-d-i+j}\subseteq\prod_{j=0}^t\ann_R\E_\infty^{d-j,-d-i+j}\subseteq\ann_R\H^{-i}$.

Now let $t=0$.
The first spectral sequence yields $\Ext_R^d(\Tor_i^R(M,N),\omega)\cong\Ext_R^{d+i}(M,N^\dag)$.
As $M$ is locally free on $\speco(R)$, the $R$-modules $\Tor_i^R(M,N),\,\Ext_R^{d+i}(M,N^\dag)$ have finite length.
By \cite[Corollary 3.5.9]{BH}, we get
$$
\Ext_R^{d+i}(M,N^\dag)^\vee\cong\Ext_R^d(\Tor_i^R(M,N),\omega)^\vee\cong\H_\m^0(\Tor_i^R(M,N))=\Tor_i^R(M,N).
$$
Hence $\ann_R\Tor_i^R(M,N)$ is equal to $\ann_R\Ext_R^{d+i}(M,N^\dag)^\vee$, which coincides with $\ann_R\Ext_R^{d+i}(M,N^\dag)$ by \cite[Proposition 3.2.12(c)]{BH}.

(2) For each $j\ge0$ there is an exact sequence $0\to\syz^{j+1}M\to R^{\oplus m_j}\to\syz^jM\to0$, which induces an exact sequence $\Ext_R^{i+j}(N,R)^{\oplus m_j}\to\Ext_R^{i+j}(N,\syz^jM)\to\Ext_R^{i+j+1}(N,\syz^{j+1}M)$.
By \cite[Theorem 2.3]{trace}, the ideal $\trace\omega$ annihilates $\Ext_R^l(X,R)$ for all $l>0$ and $X\in\cm(R)$.
It is observed that $(\trace\omega)\cdot\ann_R\Ext_R^{i+j+1}(N,\syz^{j+1}M)\subseteq\ann_R\Ext_R^{i+j}(N,\syz^jM)$, and hence
\begin{align*}
(\trace\omega)^r\cdot\ann_R\Ext_R^{r+i}(N,\syz^rM)
&\subseteq(\trace\omega)^{r-1}\cdot\ann_R\Ext_R^{r+i-1}(N,\syz^{r-1}M)\\
&\subseteq\cdots
\subseteq(\trace\omega)\cdot\ann_R\Ext_R^{i+1}(N,\syz M)
\subseteq\ann_R\Ext_R^i(N,M).
\end{align*}
It is also seen from \cite[Theorem 2.3]{trace} that $\trace\omega$ annihilates $\Ext_R^l(\omega,X)$ for all $l>0$ and $X\in\mod R$.
A dual argument using this and exact sequences $0\to \cos^j N\to \omega^{\oplus n_j}\to \cos^{j+1}N\to 0$ shows the inclusion $(\trace\omega)^r\cdot\ann_R\Ext_R^{r+i}(\cos^rN,M)\subseteq\ann_R\Ext_R^i(N,M)$.  
\end{proof}  

We can now prove the following proposition, which is an essential part of the theorem stated below.

\begin{prop}\label{6}
Let $R$ be a $d$-dimensional Cohen--Macaulay local ring with a canonical module $\omega$.
Let $\X,\Y$ be subcategories of $\mod R$ with $\Y\subseteq\cm(R)$.
Let $n\ge0$ be an integer.
\begin{enumerate}[\rm(1)]
\item
Let $0\le t\le d$ be an integer.
If $\X$ is contained in $\mod_{t}(R)$, then 
$$
(\t_n(\X,\Y))^{t+1}\subseteq\e^{d+n}(\X,\Y^\dag),\qquad
(\e^{d+n-t}(\X,\Y))^{t+1}\subseteq\t_n(\X,\Y^\dag),\qquad
$$
and the equalities hold when $t=0$.
\item
For each integer $r\ge0$ one has an inclusion
$$
(\trace\omega)^r\cdot\e^{r+n}(\Y,\syz^r\X)\subseteq \e^n(\Y,\X)\supseteq(\trace\omega)^r\cdot\e^{r+n}(\cos^r\Y,\X).
$$
\end{enumerate}
\end{prop}

\begin{proof}
(1) Fix $i>n$, $0\le j\le t$, $M\in\X$ and $N\in\Y$.
Note then that $\dim\nf(M)\le t$. 

Take any element $a_j\in\t_n(\X,\Y)$.
Since $i+j\ge i>n$, the element $a_j$ belongs to $\ann_R\Tor_{i+j}^R(M,N)$, which is contained in $\ann_R\Ext_R^{d-j}(\Tor_{i+j}^R(M,N),\omega)$.
The first inclusion in Lemma \ref{2}(1) yields
$$\textstyle
a_0\cdots a_t\in\prod_{j=0}^t\ann_R\Ext_R^{d-j}(\Tor_{i+j}^R(M,N),\omega)
\subseteq\ann_R\Ext_R^{d+i}(M,N^\dag).
$$
As we fix $i>n$, $M\in\X$ and $N\in\Y$, we get $a_0\cdots a_t\in\e^{d+n}(\X,\Y^\dag)$.
Thus $(\t_n(\X,\Y))^{t+1}\subseteq\e^{d+n}(\X,\Y^\dag)$.

The other inclusion is similarly deduced.
Pick $a_j\in\e^{d+n-t}(\X,\Y)$.
As $d+i-j\ge d+i-t>d+n-t$, we have $a_j\in\ann_R\Ext_R^{d+i-j}(M,N)\subseteq\ann_R\Ext_R^{d-j}(\Ext_R^{d+i-j}(M,N),\omega)$.
The second inclusion in Lemma \ref{2}(1) implies
$$\textstyle
a_0\cdots a_t\in\prod_{j=0}^t\ann_R\Ext_R^{d-j}(\Ext_R^{d+i-j}(M,N),\omega)\subseteq\ann_R\Tor_i^R(M,N^\dag),
$$
and hence $a_0\cdots a_t\in\t_n(\X,\Y^\dag)$.
Therefore, the inclusion $(\e^{d+n-t}(\X,\Y))^{t+1}\subseteq\t_n(\X,\Y^\dag)$ follows.

When $t=0$, the equality $\t_n(\X,\Y)=\e^{d+n}(\X,\Y^\dag)$ follows from the last assertion of Lemma \ref{2}(1).
Replacing $\Y$ with $\Y^\dag$, we see that the equality $\t_n(\X,\Y^\dag)=\e^{d+n}(\X,\Y)$ also holds\footnote{These two equalities can also be deduced from the inclusions given in the first part of Proposition \ref{6}(1). In fact, letting $t=0$ yields $\t_n(\X,\Y)\subseteq\e^{d+n}(\X,\Y^\dag)$ and $\e^{d+n}(\X,\Y)\subseteq\t_n(\X,\Y^\dag)$. Then replace $\Y$ with $\Y^\dag$.}.

(2) The assertion immediately follows from Lemma \ref{2}(2).
\end{proof}  

Here are two immediate consequences of the above proposition.

\begin{cor}\label{000}
Let $R$ be a $d$-dimensional Cohen--Macaulay local ring with a canonical module.
Let $n\ge0$ be an integer.
Then $\t_n(\cm_0(R))=\e^{d+n}(\cm_0(R))$ if $R$ is locally Gorenstein on $\speco(R)$. 
\end{cor}

\begin{proof}
Since $R$ is locally Gorenstein on the punctured spectrum, $\cm_0(R)$ is closed under canonical duals, that is to say, $(\cm_0(R))^\dag=\cm_0(R)$. 
Let $t=0$ and $\X=\Y=\cm_0(R)$ in Proposition \ref{6}(1).
\end{proof}

\begin{cor}\label{12}
Let $R$ be an artinian local ring.
Let $n\ge0$ be an integer.
\begin{enumerate}[\rm(1)]
\item
Let $\X,\Y$ be subcategories of $\mod R$.
If $\Y$ is closed under Matlis duals, then $\t_n(\X,\Y)=\e^n(\X,\Y)$.
\item
There is an equality $\t_n(\mod R)=\e^n(\mod R)$.
\end{enumerate}
\end{cor}

\begin{proof}
Letting $d=t=0$ in Proposition \ref{6}(1) yields the assertion.
\end{proof}

To state our theorem below, we recall the definition of the non-Gorenstein locus of $R$.

\begin{dfn}
We denote by $\ng(R)$ the {\em non-Gorenstein locus} of $R$, that is, the set of prime ideals $\p$ of $R$ such that the local ring $R_\p$ is non-Gorenstein.
If $R$ is a Cohen--Macaulay ring with a canonical module $\omega$, it holds that $\ng(R)=\nf(\omega)=\V(\trace\omega)$ (see \cite[Lemma 2.1]{HHS}).
\end{dfn}

Now we can state and prove the theorem below.
The second assertion contains
part of Theorem \ref{14}(2).   

\begin{thm}\label{3}
Let $R$ be a $d$-dimensional Cohen--Macaulay local ring with a canonical module $\omega$.
Let $n\ge0$ and $0\le t\le d$ be integers. 
\begin{enumerate}[\rm(1)]
\item
Let $\X$ be a subcategory of $\cm_{t}(R)$. 
Let $\Y$ be a subcategory of $\cm(R)$ closed under canonical duals.
Assume either that $\Y$ is closed under syzygies  or that $\X$ is closed under cosyzygies.  
Then 
$$
(\trace\omega)^d\cdot(\t_n(\X,\Y))^{t+1}\subseteq\e^n(\X,\Y),\qquad(\e^n(\X,\Y))^{t+1}\subseteq\t_n(\X,\Y). 
$$
\item
If $\dim\ng(R)\le t$ and $\cm_0(R)\subseteq \X \subseteq \cm_t(R)$, then $\sqrt{\t_n(\X,\cm_{t}(R))}=\sqrt{\e^n(\X,\cm_{t}(R))}$.
\end{enumerate}
\end{thm}  

\begin{proof}
(1) First of all, since $\Y$ is closed under canonical duals, we have $\Y=\Y^\dag$.

The following argument deduces the first inclusion in the assertion.
\begin{align*}
(\trace\omega)^d\cdot(\t_n(\X,\Y))^{t+1}
&\overset{\rm(a)}{\subseteq} (\trace\omega)^d\cdot\e^{d+n}(\X,\Y)\\
&\overset{\rm(b)}{\subseteq}
\begin{cases}
(\trace\omega)^d\cdot\e^{d+n}(\X,\syz^d\Y)&\text{if $\Y$ is closed under syzygies},\\
(\trace\omega)^d\cdot\e^{d+n}(\cos^d\X,\Y)&\text{if $\X$ is closed under cosyzygies}
\end{cases}
\overset{\rm(c)}{\subseteq}\e^n(\X,\Y). 
\end{align*} 

Here are the reasons why (a)--(c) hold.\ 
(a): Applying Proposition \ref{6}(1), we have an inclusion $(\t_n(\X,\Y))^{t+1}\subseteq \e^{d+n}(\X,\Y)$.
(b):  If $\Y$ (resp. $\X$) is closed under syzygies (resp. cosyzygies), then $\syz^d\Y$ (resp. $\cos^d\X$) is contained in $\Y$ (resp. $\X$), and hence $\e^{d+n}(\X,\Y)$ is contained in $\e^{d+n}(\X,\syz^d\Y)$ (resp. $\e^{d+n}(\cos^d\X,\Y)$).
(c): Use Proposition \ref{6}(2). 

Next we show the second inclusion in the assertion.
Since $d-t\ge0$, the ideal $\e^n(\X,\Y)$ is contained in the ideal $\e^{d+n-t}(\X,\Y)$, whose $(t+1)$st power is contained in $\t_n(\X,\Y)$ by Proposition \ref{6}(1).

(2)  When $d=0$, we have $t=0$ and $\X=\cm_t(R)=\mod R$.
By Corollary \ref{12}(2) the assertion holds.
Hence we assume $d>0$, so that we get equalities $\V((\trace \omega )^d)=\V(\trace \omega )=\ng (R)$.

We claim that for each maximal Cohen--Macaulay $R$-module $M$ there is an inclusion
$$
\nf(M^\dag)\subseteq\nf(M)\cup\ng(R).
$$
Indeed, if $\p$ is a prime ideal of $R$ which does not belong to $\nf(M)\cup\ng(R)$, then both $M_\p$ and $\omega_\p$ are $R_\p$-free, and so is $(M^\dag)_\p$.

Combining this claim with the assumption $\dim\ng(R)\le t$, we observe that $\cm_t(R)$ is closed under canonical duals.
The subcategory $\cm_t(R)$ is always closed under syzygies.
Applying (1) to $\X$ and $\Y:=\cm_t(R)$, we get the first line below, which yields the second.
\begin{align*}
&(\trace\omega)^d\cdot(\t_n(\X,\cm_t(R)))^{t+1}\subseteq\e^n(\X,\cm_t(R)),\quad(\e^n(\X,\cm_t(R)))^{t+1}\subseteq\t_n(\X,\cm_t(R)).\\
&\V(\e^n(\X,\cm_t(R)))\subseteq\ng(R)\cup\V(\t_n(\X,\cm_t(R))),\quad\V(\t_n(\X,\cm_t(R)))\subseteq\V(\e^n(\X,\cm_t(R))).
\end{align*}
There are inclusions $\ng(R)\subseteq \sing R\subseteq\V(\t_n(\cm_0(R)))\subseteq\V(\t_n(\X,\cm_t(R)))$, where the second one follows from Lemma \ref{3.11}(3).
Hence $\V(\e^n(\X,\cm_t(R)))$ is contained in $\V(\t_n(\X,\cm_t(R)))$.
Now we obtain the equality $\V(\e^n(\X,\cm_t(R)))=\V(\t_n(\X,\cm_t(R)))$, which completes the proof.
\end{proof}

A natural question arises.

\begin{ques}
Let $R$ be a Cohen--Macaulay local ring.
Does the following equality always hold?
\begin{equation}\label{11}
\t_0(\cm_0(R))=\e^0(\cm_0(R)).
\end{equation}
\end{ques}

Theorem \ref{3}(2) and Corollaries \ref{13}, \ref{12} guarantee that \eqref{11} holds if $R$ is either Gorenstein or artinian, and holds up to radical when $R$ is locally Gorenstein on the punctured spectrum.
Corollary \ref{000} says that, when $R$ is locally Gorenstein on the punctured spectrum, \eqref{11} is equivalent to the equality $\e^0(\cm_0(R))=\e^d(\cm_0(R))$, where $d=\dim R$.
We do not know any counterexample to \eqref{11}.
Below, we give a simple class of examples which supports \eqref{11}. 

\begin{ex}
Let $k$ be a field, and consider the numerical semigroup ring $R=k[\![t^n,t^{n+1},\dots,t^{2n-1}]\!]$ with $n\ge1$.
Then $R$ is a $1$-dimensional local domain.
In particular, $R$ is Cohen--Macaulay and satisfies $\cm_0(R)=\cm(R)$.
Also, $R$ is non-Gorenstein if $n\ge3$, and has infinite Cohen--Macaulay representation type if $n\ge4$ by \cite[Theorem 4.10]{LW}.
The conductor of $R$ clearly coincides with the maximal ideal $\m$ of $R$, which is contained in $\e^0(\cm(R))$ by \cite[Proposition 3.1]{W}.
Note that $\t_0(\cm(R))\ne R$ unless $R$ is regular.
By Proposition \ref{7}, we get $\e^0(\cm(R))=\t_0(\cm(R))=\m$ for all $n\ge2$.
\end{ex}

We denote by $\sing R$ the {\em singular locus} of $R$, that is, the set of prime ideals $\p$ of $R$ such that $R_\p$ is singular.
Note that $R$ has an isolated singularity if and only if $\dim\sing R\le0$.

To give our final result in this paper, we need to verify that \cite[Theorem 5.11(1)]{radius} holds for an arbitrary excellent Cohen--Macaulay local ring containing a field and admitting a canonical module, and that the implication (b) $\Rightarrow$ (a) in \cite[Theorem 1.1(1)]{dim} holds even if we replace $\e^0(\cm_0(R))$ with $\e^n(\cm_0(R))$.  

\begin{prop}\label{20}
\begin{enumerate}[\rm(1)]
\item
Let $R$ be a $d$-dimensional excellent equicharacteristic local ring.
\begin{enumerate}[\rm(a)]
\item
The subcategory $\syz^{2d}(\mod R)$ of $\mod R$ has finite size and radius.
\item
Suppose that $R$ is Cohen--Macaulay and admits a canonical module.
Then the subcategory $\cm(R)$ of $\mod R$ has finite rank, size, dimension and radius.
\end{enumerate}
\item
Let $(R,\m,k)$ be a $d$-dimensional local Cohen--Macaulay ring.
Let $n\ge0$ be an integer.
Suppose that the ideal $\e^n(\cm_0(R))$ of $R$ contains some power of $\m$.
Then $R$ has an isolated singularity, and $\cm_0(R)=\cm(R)$ has finite rank, size, dimension and radius. 
\end{enumerate} 
\end{prop}

\begin{proof}
(1a) We modify the proof of \cite[Theorem 5.7]{radius}.
Replace ``\,$d$\,'', ``\,$d-1$\,'' and ``\,$\syz_RM$\,'' with ``\,$2d$\,'', ``\,$2(d-1)$\,'' and ``\,$\syz_R^2M$\,'' respectively.
Then it proves the assertion.
Indeed, the assumption of \cite[Theorem 5.7]{radius} that $R$ is complete and has perfect coefficient field is used only to apply \cite[Corollary 5.15]{W} to find an ideal $J$ of $R$ which satisfies $\sing R=\V(J)$ and annihilates $\Ext_R^{d+1}(M,N)$ for all $R$-modules $M,N$ (in the case where $R$ is a singular domain with $d>0$).
By virtue of \cite[Theorem 5.3]{ua}, there is an ideal $J'$ of $R$ which satisfies $\sing R=\V(J')$ and annihilates $\Ext_R^{2d+1}(M,N)$ for all $R$-modules $M,N$. 

Here is the flow of the proof.
We use induction on $d$, and the case $d=0$ follows by the original argument.
Let $d>0$.
As in the original argument, we may assume that $R$ is a singular domain.
Take $J'$ as above, and find an element $0\ne x\in J'$.
Put $N=\syz_R^{2d}M$.
Then $x$ is $N$-regular.
As in the original argument, $N$ is a direct summand of $\syz_R(N/xN)$, and $N/xN\cong\syz_{R/xR}^{2(d-1)}(\syz_R^2M/x\syz_R^2M)\in\syz_{R/xR}^{2(d-1)}(\mod R/xR)$.
 
(1b) We modify the proof of \cite[Corollary 5.9]{radius}.
Replace ``\,$d$\,'' with ``\,$2d$\,'' in it, and apply (1a) instead of \cite[Theorem 5.7]{radius}.
Then we observe that \cite[Corollary 5.9]{radius} remains valid for any excellent equicharacteristic local ring.
Combining this with \cite[Proposition 5.10]{radius} deduces the assertion.

(2)
(i) We first deal with the case where $R$ admits a canonical module.
Using Lemma \ref{3.11}(3), we have $\sing R\subseteq\V(\e^n(\cm_0(R)))\subseteq\{\m\}$.
This particularly says that $R$ has an isolated singularity, and hence $\cm_0(R)=\cm(R)$.
Similarly as in the proof of \cite[Proposition 6.1(2a)]{dim}, we observe that there exists an integer $r>0$ such that $\syz^n(\cm(R))\subseteq|\syz^dk|_r$.
An analogous argument as in the proof of \cite[Corollary 5.9]{radius} yields that $\cm(R)\subseteq|\syz^dk\oplus W|_{r(n+1)}$ for some $W\in\cm(R)$.
We obtain $\cm(R)=|\syz^dk\oplus W|_{r(n+1)}$, which shows that $\cm(R)$ has rank less than $r(n+1)$.

(ii) Now, let us handle the general case when $R$ may not possess a canonical module.
By assumption, $\e^n(\cm_0(R))$ contains $\m^h$ for some $h>0$.
Fix $i>n$ and $X,Y \in \cm_0(\widehat R)$, where $\widehat R$ is the completion of $R$.
By \cite[Corollary 3.3]{stcm}, there exist $M,N \in \cm_0(R)$ such that $X,Y$ are direct summands of $\widehat M, \widehat N$ respectively.
Hence $\Ext^i_{\widehat R}(X,Y)$ is a direct summand of $\Ext^i_{\widehat R} (\widehat M, \widehat N)$, which is annihilated by $(\m\widehat R)^h$.
Thus $\e^n(\cm_0(\widehat R))$ contains $(\m\widehat R)^h$.
As $\widehat R$ admits a canonical module, it follows from (i) that $\widehat R$ has an isolated singularity, and $\cm_0(\widehat R)=\cm(\widehat R)$ has finite rank, size, dimension and radius.
It is an elementary fact that having an isolated singularity descends from $\widehat R$ to $R$, and we have $\cm_0(R)=\cm(R)$.
An argument similar to the one done at the end of \cite[Theorem 5.11(2)]{radius} shows that $\cm(R)$ has finite rank.
Hence $\cm(R)$ has finite dimension, size and radius as well.
\end{proof}  

\begin{rem}
Proposition \ref{20}(1a) refines the latter statement of \cite[Theorem 5.3]{ua}, which asserts that the subcategory $\syz^{3d}(\mod R)$ of $\mod R$ has finite size.
\end{rem}

As an application of Theorem \ref{3}, one can refine the main results of \cite{dim}.
More precisely, it is asserted in \cite[Theorem 1.1 and Corollary 7.2]{dim} that among the four conditions (a)--(d) given in the corollary below,
\begin{enumerate}[\quad(A)]
\item
the implication ${\rm(a)}\Rightarrow{\rm(d)}$ holds,
\item
the implications ${\rm(a)} \Leftrightarrow {\rm(b)} \Rightarrow {\rm(c)} \Rightarrow {\rm(d)}$ hold for $t=0$, and 
\item
the equivalences ${\rm(a)} \Leftrightarrow {\rm(b)} \Leftrightarrow {\rm(c)} \Leftrightarrow {\rm(d)}$ hold for $t=0$ provided that $R$ is complete and has perfect coefficient field.
\end{enumerate}
When $R$ admits a canonical module, by using Theorem \ref{3}(2) we can improve the above statements (A), (B), (C) as follows. 

\begin{cor}\label{8}
Let $R$ be a $d$-dimensional Cohen--Macaulay local ring. 
Let $n\ge 0$ and $0\le t\le d$ be integers. 
Consider the following conditions.
\begin{align*}
{\rm(a)}\ \dim\cm_t(R)<\infty.\quad
&{\rm(b)}\ \dim\V(\e^n(\cm_t(R)))\le t.\\ 
&{\rm(c)}\ \dim\V(\t_n(\cm_t(R)))\le t. \quad
{\rm(d)}\ \dim\sing R\le t.
\end{align*}
\begin{enumerate}[\rm(1)]
\item
The implications ${\rm(a)} \Rightarrow {\rm(b)} \Rightarrow {\rm(c)} \Rightarrow {\rm(d)}$ hold.  
\item
The equivalence ${\rm(b)}\Leftrightarrow{\rm(c)}$ holds when $R$ admits a canonical module.
\item
The implications ${\rm(a)}\Leftrightarrow{\rm(b)}\Leftrightarrow{\rm(c)}\Rightarrow{\rm(d)}$ hold when $t=0$.   
\item
The equivalences ${\rm(a)}\Leftrightarrow{\rm(b)}\Leftrightarrow{\rm(c)}\Leftrightarrow{\rm(d)}$ hold when $R$ is excellent, equicharacteristic and admits a canonical module.
\end{enumerate} 

\end{cor}

\begin{proof} 
(1) It follows from Proposition \ref{7} that (b) implies (c).
The assertion that (c) implies (d) follows from the inclusions $\sing R\subseteq\V(\t_n(\cm_0(R)))\subseteq\V(\t_n(\cm_t(R)))$ by Lemma \ref{3.11}(3).
Let us prove that (a) implies (b).
Note that the set $\V(\e^n(\cm_t(R)))$ is contained in $\V(\e^0(\cm_t(R)))$.
It suffices to show that $\V(\e^0(\cm_t(R)))$ has dimension at most $t$.
We make a similar argument as in the proof of \cite[Proposition 6.1(1a)]{dim}.  
Write $\cm_t(R)=[G]_r$ with $G\in\cm_t(R)$ and $r>0$.  
We have
$$
\V(\e^0(\cm_t(R)))=\V(\e^0(G,\cm_t(R)))\subseteq\V(\e^0(G,\mod R))=\nf(G)
$$
by \cite[Lemma 5.3(1) and Proposition 5.1(1)]{dim}.
It follows that the dimension of $\V(\e^0(\cm_t(R)))$ is at most the dimension of $\nf(G)$, which is at most $t$ since $G\in\cm_t(R)$.

(2) Suppose that (b) (resp. (c)) holds.
Then it follows from (1) that (d) holds, i.e., $\dim \sing R\le t$.
Since $\ng(R)\subseteq\sing R$, we get $\dim \ng(R)\le t$.
Hence (c) (resp. (b)) holds by Theorem \ref{3}(2).
  
(3) Let $t=0$.
Proposition \ref{20}(2) shows that (b) implies (a).
Due to (1), we only need to prove that (c) implies (b).
So assume that (c) holds.
Then $\t_n(\cm_0(R))$ contains some power of $\m$.
A similar argument as in (ii) in the proof of Proposition \ref{20}(2) done for $\Tor$ instead of $\Ext$ shows that $\t_n(\cm_0(\widehat R))$ contains a power of the maximal ideal $\m\widehat R$ of the completion $\widehat R$ of $R$.
Lemma \ref{3.11}(3) implies that $\widehat R$ has an isolated singularity, and so does $R$.
Hence $\cm_0(\widehat R)=\cm(\widehat R)$ and $\cm_0(R)=\cm(R)$.
Since $\widehat R$ admits a canonical module, it follows from (2) that $\e^n(\cm(\widehat R))$ contains $(\m\widehat R)^h$ for some $h>0$.
Then it is easy to observe that $\e^n(\cm(R))$ contains $\m^h$, and (b) follows.

(4) Suppose that $R$ is excellent and equicharacteristic. 
Then Proposition \ref{20}(1b) shows $\dim\cm(R)<\infty$. 
If (d) holds, then $\cm_t(R)=\cm(R)$ and (a) follows.
Combining this with (1) completes the proof.
\end{proof} 

\begin{rem}\label{10}
\begin{enumerate}[(1)]
\item
In view of (3) of Corollary \ref{8}, it is natural to ask whether (b) implies (a) for $t>0$.
It seems to be quite nontrivial even in the case $t=d>0$.
Indeed, in this case, (b) automatically holds and $\cm_t(R)=\cm(R)$.
We do not know in general whether $\cm(R)$ has finite dimension when $R$ is not equicharacteristic, even if we assume that $R$ is complete and has an isolated singularity. 
\item
Making an analogous argument as in its proof, one actually obtains a more general statement than Corollary \ref{8}(1):
\begin{quote}
Let $\X,\Y$ be subcategories of $\mod R$.
Suppose that $\X$ has finite dimension and is contained in $\mod_t(R)$.
Then $\V(\t_0(\X,\Y))$ and $\V(\e^0(\X,\Y))$ have dimension at most $t$.
\end{quote}
This is a generalization of \cite[Proposition 6.1(1a)]{dim}; letting $t=0$ recovers it. 
\item
Let $(R,\m)$ be a local ring not having an isolated singularity, and let $n=0$ and $\X=\Y=\cm_0(R)$.
Then no nonmaximal prime ideal of $R$ belongs to the supports of the modules $\ttt_n(\X,\Y)$ and $\eee^n(\X,\Y)$, while neither $\t_n(\X,\Y)$ nor $\e^n(\X,\Y)$ is $\m$-primary by Corollary \ref{8}(1).
Hence we get strict inclusions
$$
\supp_R(\ttt_n(\X,\Y))\subsetneq\V(\t_n(\X,\Y)),\qquad
\supp_R(\eee^n(\X,\Y))\subsetneq\V(\e^n(\X,\Y)).
$$
\end{enumerate}
\end{rem} 

\begin{ac}
The authors sincerely thank the anonymous referee for carefully reading the paper and giving numerous valuable suggestions.
\end{ac}

\end{document}